   \documentclass[11pt,letterpaper]{amsart}
\usepackage{amsmath, amssymb, amsfonts, graphicx, pgfplots, tikz, tikz-cd, stmaryrd}

\newtheorem{Theorem}{Theorem}[section]
\newtheorem{prop}[Theorem]{Proposition}
\newtheorem{Lemma}[Theorem]{Lemma}

\newtheorem{Remark}[Theorem]{Remark}
\newtheorem{Corollary}[Theorem]{Corollary}

\def\beq#1#2\eeq{%
        \begin{equation}%
        \label{#1}%
            #2%
        \end{equation}%
   }

\usepackage{color}
\usepackage{graphicx}
\usepackage{epstopdf}

\title[Graph-associahedra and inversion]{Differential algebra of polytopes and inversion formulas}

\author{V.M. Buchstaber}\address{Steklov Mathematical Institute and Moscow State University, Russia}
\email{buchstab@mi-ras.ru}

\author{A.P. Veselov}
\address{Department of Mathematical Sciences,
Loughborough University, Loughborough LE11 3TU, UK}
\email{A.P.Veselov@lboro.ac.uk}

\begin{document}

\maketitle

\begin{abstract}
 We use the differential algebra of polytopes to explain the known remarkable relation of the combinatorics of the associahedra and permutohedra with the universal compositional and multiplicative inversion formulas for the formal power series. This approach allows to single out the associahedra and permutohedra among all graph-associahedra and emphasizes the significance of the differential equations for special sequences of simple polytopes derived earlier by one of the authors. We discuss also the link with the geometry of Deligne-Mumford moduli spaces $\bar M_{0,n}$ and permutohedral varieties, as well as the interpretation of the combinatorics of cyclohedra in relation with the classical Fa\`a di Bruno's formula.
  \end{abstract}


\section{Introduction}

The following natural question goes back at least to Lagrange.
For a given power series 
$\alpha(u)=u+\sum_{n\geq 1}a_n u^{n+1}$ find the inversion $\beta(v)$ under the substitution: $\alpha(\beta(v))\equiv v, \,\, \beta(\alpha(u))\equiv u.$
The coefficients of such series $\beta(v)=v+\sum_{n\geq 1}b_n v^{n+1}$ are known to be certain polynomials with integer coefficients:
$$
b_1=-a_1, \, b_2=-a_2+2a_1^2,\,\, b_3=-a_3+5a_1a_2-5a_1^3,
$$
$$
b_4=-a_4+6a_1a_3+3a_2^2-21a_1^2a_2+14a_1^4.
$$

Remarkably these coefficients can be interpreted in terms of the combinatorics of the {\it  associahedra} (or, {\it Stasheff polytopes}), see \cite{AA-2023, Loday} for the details.
For example, the formula for $b_4$ says that $3D$ associahedron has 6 pentagonal and 3 quadrilateral faces, 21 edges and 14 vertices (see Fig. 1, where we show also its realisation as a truncated cube following \cite{BV}).

\begin{figure}[h]
\includegraphics[width=35mm]{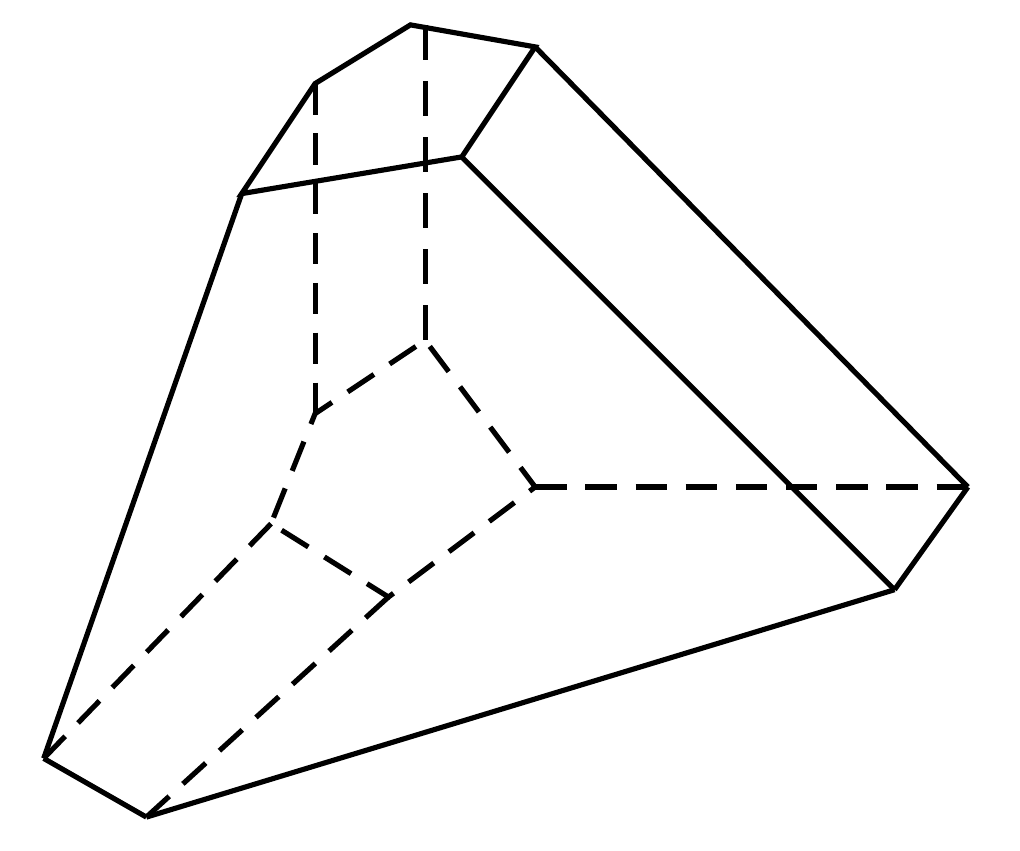}  \quad \includegraphics[width=30mm]{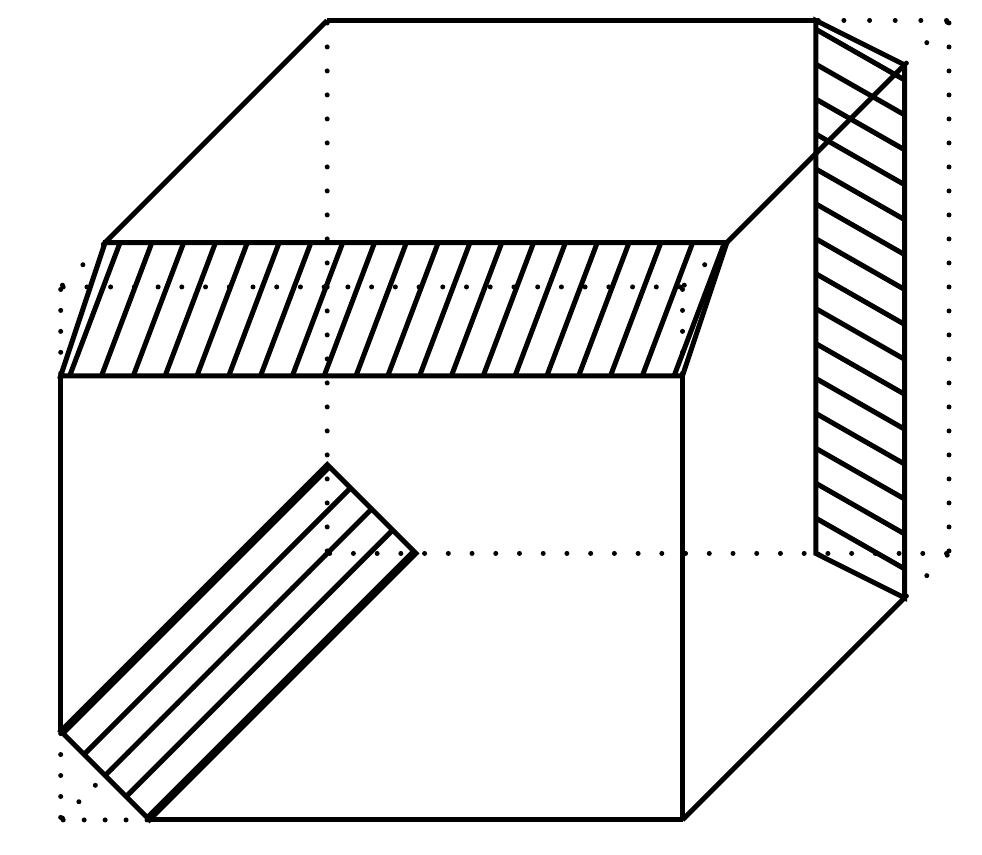} 
\caption{3D associahedron and its realisation as a truncated cube.}
\end{figure}


It is difficult to trace who was the first to observe this remarkable fact. The first combinatorial description of the coefficients $b_n$ seems to belong to Raney \cite{Raney}. Stanley \cite{Stanley} provided 3 proofs in terms of tree combinatorics, known to be closely related to the associahedron. Loday \cite{Loday} described this explicitly in terms of the associahedra, see also recent book by Aguiar and Ardila \cite{AA-2023}, who used the theory of Hopf monoids to derive this connection. 

The aim of this paper is to provide one more derivation of this link using the differential equations for the generating function of the associahedra derived by one of the authors \cite{B-2008}. This allows us also to prove another known remarkable formula, expressing the multiplicative inversion of a formal power series in terms of combinatorics of permutohedra (see \cite{AA-2023}) and to provide an interpretation of the combinatorics of cyclohedra in relation with the classical Fa\`a di Bruno's formula, which was first discovered by Aguiar and Bastidas \cite{Bast}. In the case of the stellohedra we derive the relation with permutohedra, which seems to be new.

We discuss also the link with the rich theory of Deligne-Mumford moduli spaces $\bar M_{g,n}$, which are known to be related to associahedra \cite{Kap}, inversion formulas \cite{McMullen} and to the KdV hierarchy \cite{Witten}.

\section{Differential algebra of simple polytopes}

We start with a brief description of the differential algebra of simple polytopes following \cite{B-2008, BP}. For the general theory of convex polyhedra we refer to Ziegler \cite{Ziegler}.

Let $\mathfrak P:=\sum_{n\geq 0} \mathfrak P^n$ be the free abelian group generated by all convex polytopes, considered up to combinatorial equivalence and naturally graded by the dimension of the polytopes.
The product of polytopes turns $\mathfrak P$ into a graded commutative ring with the unit given by the point, considered as $0$-dimensional polytope. Simple polytopes form a graded subring $\mathfrak S \subset \mathfrak P.$ Note that the ring $\mathfrak P$ is different both from the polytope algebra \cite{P_Mc2} of convex polytopes in $\mathbb R^n$ with product given by the Minkowski sum and from the Grothendieck ring of $\Gamma$-rational polytopes considered in \cite{Nicaise}.

A polytope is called {\it indecomposable} if it cannot be represented as a product of polytopes of positive dimension.

\begin{Theorem} [\cite{BE}]
The ring $\mathfrak P$ is a polynomial ring generated by indecomposable combinatorial polytopes, so any combinatorial polytope can be uniquely represented as a product of the indecomposable polytopes.
\end{Theorem}
The proof can be found in \cite{BE} (see also Prop.1.7.2 in \cite{BP}). Note that if we introduce the second grading in $\mathfrak P$ by the number of the facets of polytopes, then the number of generators in any bi-graded component of $\mathfrak P$ will be finite \cite{BP}.


Following \cite{B-2008} introduce the {\it derivation} $d$ in $\mathfrak P$ by defining the boundary $dP$ of a given polytope $P$ simply as the sum of all facets of $P$ considered as elements in $\mathfrak P$. It is easy to check that $d$ preserves the subring $\mathfrak S \subset \mathfrak P$ and satisfies the Leibnitz identity in $\mathfrak P$
$$
d(P_1P_2)=(dP_1)P_2+P_1(dP_2).
$$
This supplies $\mathfrak P$ with a structure of the differential ring, with  $\mathfrak S$ being its differential subring. Note that in our case $d^2\neq 0$ in contrast with the differential introduced in \cite{Gaif}.

We will be interested in the special class of simple polytopes known as {\it graph-associahedra}.
They were introduced and studied in the work of Carr and Devadoss \cite{CD} and independently by Toledano Laredo \cite{TL} under the name De Concini–Procesi associahedra (see also Postnikov \cite{Post}). This class contains the classical series of permutohedra and associahedra.

Let $\Gamma$ be a connected simple (no loops or multiple edges) graph with the set of $n$ vertices, which we identify with $[n]:=\{1,2,\dots, n\}.$
The corresponding graph-associahedron $P_\Gamma$ is a particular case of the nestohedron \cite{FS} with the building set $\mathcal B=\mathcal B(\Gamma)$ consisting of the connected induced subgraphs of $\Gamma$, or equivalently as the subsets $S\subset [n]$ such that the restriction $\Gamma|_S$ is connected.
Explicitly $P_\Gamma$ can be realised as the following $(n-1)$-dimensional convex polytope
\beq{PG}
P_\Gamma=\{x\in \mathbb R^{n}:  \sum_{i=1}^{n}x_i= 3^{n}, \,\, \sum_{i\in S}x_i\geq 3^{|S|},\, S\in \mathcal B, \, S\neq [n]\},
\eeq
where $|S|$ is the number of elements in $S$ (see \cite{TL}).

The boundary of the graph-associahedra can be given by the following formula (see e.g. \cite{BP}, Prop. 1.7.16).
For any connected induced subgraph $S$ of $\Gamma$ with vertex set $[S]\subset [n]$ define the graph $\Gamma/S$ with the vertex set $[n]\setminus [S]$ having an edge between two vertices $i$ and $j$ whenever they are path-connected in the restriction $\Gamma|_{S\cup\{i,j\}}.$ Then the boundary of the graph-associahedron $P_\Gamma$ can be given combinatorially as 
\beq{boundary}
dP_\Gamma=\sum_{S \in \mathcal B(\Gamma), \, S \neq \Gamma}P_{S}\times P_{\Gamma/S}.
\eeq

\begin{prop}
For any connected graph $\Gamma$ the corresponding graph-assiciahedron $P_\Gamma$ is indecomposable.
\end{prop}

\begin{proof}
We can prove this by induction in the number $k$ of the vertices of $\Gamma$. For $k=1,2$ this is obvious. Let this be true for all $k<n$ and consider $P_\Gamma$ for $\Gamma$ with $n$ vertices. Assume that  $P_\Gamma=P_1\times P_2$ is decomposable, so that the boundary $dP_\Gamma=dP_1\times P_2+P_1\times dP_2.$

On the other hand the boundary can be given by formula (\ref{boundary}), so
$$
dP_1\times P_2+P_1\times dP_2=\sum_{S \in \mathcal B(\Gamma), \, S \neq \Gamma}P_{S}\times P_{\Gamma/S}.
$$
By induction in the right-hand side we have the products of the indecomposable polytopes, containing $P_{S}\times P_{\Gamma/S}=P_{\Gamma/S}$ for any one-vertex $S$. From the uniqueness part of Theorem 2.1 it follows that either $P_1$ or $P_2$ must be a segment, which easily leads to a contradiction.
\end{proof}

There are two famous particular cases of graph-associahedra: associahedra $\mathfrak a_n$ (or Stasheff polytopes, traditionally denoted as $K_{n+1}$) and permutohedra $\pi_n$, corresponding to path and complete graphs with $n$ vertices respectively.

For the associahedron $\mathfrak a_n$ the corresponding graph $\Gamma$ is a path with edges $(k,k+1), \, 1\leq k \leq n$. The connected subgraphs $S$ are strings $[i,j]=(i, i+1,\dots, j-1, j)$, which can be naturally labelled by the brackets in the product
$x_1x_2\dots x_{i-1}(x_i\dots x_j)x_{j+1}\dots x_{n+1}$, linking this with the original Stasheff's description (see \cite{PSZ, Ziegler} for the history of this remarkable polytope and its role in combinatorics and algebra). The number of vertices of $\mathfrak a_n$ is the Catalan number $C_{n}=\frac{1}{n+1}{ 2n \choose n}.$ One can also describe $\mathfrak a_n$ also as the Newton polytope of the polynomial 
$\prod_{1\leq i<j\leq n}(x_i+x_{i+1}+\dots+x_{j-1}+x_{j}).$ 
For $n=3$ we have pentagon, for $n=4$ 3D associahedron, which is shown on Fig. 1 together with its realisation as 2-truncated cube found in \cite{BV}.

Permutohedron $\pi_n$ corresponds to the complete graph with $n$ vertices and can be described as the convex hull of the points $\sigma(\rho)\in \mathbb R^{n}, \, \sigma \in S_{n}$ with $\rho=(1,2,\dots,n),$ or as the Newton polytope of the Vandermonde polynomial
$\prod_{1\leq i<j\leq n}(x_i-x_j).$ 
For $n=3$ we have hexagon, for $n=4$ - the truncated octahedron shown on Fig. 2.

\begin{figure}[h]
\includegraphics[width=25mm]{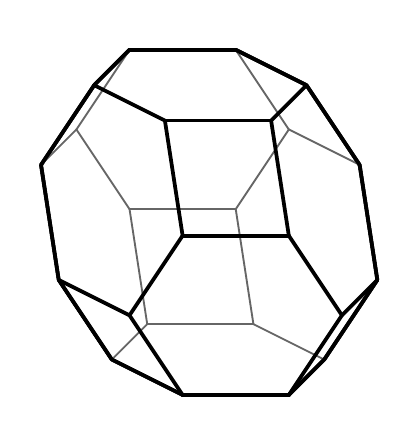} 
\caption{3D permutohedron $\pi_4$.}
\end{figure} 

Our first observation is that these two sequences of combinatorial polytopes can be characterised in terms of the differential algebra of polytopes as follows.

Let $P_k, \, k\geq 0$ be a sequence of simple polytopes with the dimension of $P_k$ equal to $k.$ Consider the following question: when the corresponding subring of $\mathfrak S$ generated by these polytopes is closed under differentiation $d$ (and thus is a differential subring of $\mathfrak S$)?

Assuming that $P_k=P_{\Gamma_{k+1}}$ are graph-associahedra of the connected graphs $\Gamma_{k+1}$ with $k+1$ vertices, we can give the following answer.

\begin{Theorem} 
A sequence of graph-associahedra $P_k, \, k\geq 0$ generates a subring of $\mathfrak S$ closed under differentiation $d$ only in two cases: associahedra $P_k=\mathfrak a_{k+1}$ and permutohedra $P_k=\pi_{k+1}.$
\end{Theorem}

\begin{proof}
%
%

Let $\Gamma_{n}, \, n\geq 1$ be the sequence of connected graphs, such that the corresponding graph-associahedra $P_k=P_{\Gamma_{k+1}}, \, k\geq 0$ (which are indecomposable by Proposition 2.2) freely generate a polynomial subring of $\mathfrak S$, which is closed under differentiation $d.$

We need the following result from graph theory.

\begin{Lemma} 
Let $\Gamma_n, \, n \in \mathbb N$ be a sequence of connected graphs with $n$ vertices, such that any connected induced subgraph $S$ of $\Gamma_n$ and the corresponding graph $\Gamma_n/S$ are equivalent to some of $\Gamma_j$ with $j<n.$ Then this must be the sequence of paths or complete graphs.
\end{Lemma}


\begin{proof} We will prove this by induction in $n.$ For $n\leq 3$ the claim is obvious since $\Gamma_3$ is either a path or complete graph $K_3$ with 3 vertices. Assume now that the claim is true for $n\leq k$ and prove it for $n=k+1.$

Assume first that all induced subgraphs of $\Gamma=\Gamma_n$ are paths. Then $\Gamma$ must be either a tree, or a cycle. Indeed, if $\Gamma$ is not a tree, it must contain a cycle $C \subseteq\Gamma_n,$ so the corresponding induced subgraph $\bar C$ is not a path. The only exception is when $C=\Gamma,$ so $\Gamma$ is a cycle.

Let $\Gamma$ be a tree. If $\Gamma$ has a vertex $V$ of index $m\geq 3,$ then $\Gamma/V$ contains a complete graph $K_m$, which is impossible.
So all vertices of $\Gamma$ have index 1 or 2 and $\Gamma$ is a path.
If $\Gamma$ is a cycle with $n>3$ vertices, then for its subpath $S$ with $n-3$ vertices the graph $\Gamma/S=K_3$, which contradicts the assumption.
Thus $\Gamma_n$ must be a path in this case.

Assume now that all induced subgraphs of $\Gamma$ are complete graphs. If $\Gamma$ has two vertices $V_1$ and $V_2$ which are not connected by an edge, then
the shortest path connecting them is an induced subgraph, which is not complete. Thus $\Gamma_n$ must be a complete graph in this case.  
\end{proof}

The theorem now follows from Lemma and the boundary formula (\ref{boundary}) since the paths and complete graphs correspond to the associahedra and permutohedra respectively.
\end{proof}

\begin{Remark}
Note that there are other sequences of the simple polytopes $P_k$ generating a subring in $\mathfrak S$ invariant under differentiation (e.g. simplices $\Delta_k$, cubes $\Box_k$), but the general problem of description of all such sequences seems to be open. We conjecture that in the indecomposable case the answer is given by the sequences of simplices, associahedra and permutohedra.
\end{Remark}

\section{PDEs for graph-associahedra and inversion formulas}

We start with the differential equations for the generating functions of some special polytopes, derived by one of the authors \cite{B-2008} (see also \cite{BK} and \cite{BP}, section 1.8).

Let $P$ be a convex polytope of dimension $n$ and $\mathcal F_k(P)$ be the set of all faces of $P$ of  codimension $k$ considered as elements of $\mathfrak P$. 
Consider the corresponding generating function $F_P(t) \in \mathfrak P[t]$ defined by
\beq{F}
F_P(t)=\sum_{k=0}^n t^k \sum_{f\in \mathcal F_k(P)} f= \sum_{f\in \mathcal F(P)} f \, t^{\text{codim} f}.
\eeq
One can check that the valuation map $F_{t_0}: \mathfrak P \to \mathfrak P$ sending $P$ to $F_P(t_0)$ is a ring-homomorphism for all $t_0$.

Define now the  {\it motivic interior} $P^\circ$ of a polytope $P$ as the alternating sum of all the faces of $P:$
\beq{motivic}
P^\circ:=F_P(-1)=\sum_{f\in \mathcal F(P)}  (-1)^{\text{codim} f} f \in \mathfrak P.
\eeq

At the level of the set-theoretical characteristic functions this formula agrees with a well-known result in the theory of finitely additive measures of convex polytopes  \cite{P_McMullen,PukhKhov}
with $P^\circ=relint(P)$ being the standard {\it relative interior} \cite{Ziegler} of the polytope $P$. Note that the polytope $P$ can be represented simply as
$$
P=\sqcup_{f\in \mathcal F(P)}  relint(f),
$$
where $\sqcup$ denotes the disjoint union of sets (see \cite{Ziegler}). More justification for the terminology is given by the theory of toric varieties, see Section 5 below.

For simple polytopes we have the following formula, which will be important for us.

\begin{Lemma}
For any simple convex polytope $P$ we have
\beq{Lemma}
F_P(t)=e^{td}P:=\sum_{k\geq 0}\frac{t^k}{k!} d^kP,
\eeq
which can be used to characterise the simple polytopes among all polytopes.
\end{Lemma}
The proof easily follows from the simplicity of the polytope. In particular, for simple polytope $P$ the  motivic interior  can be written as
$
P^\circ=e^{-d}P.
$

Following \cite{B-2008}, consider now the generating functions of $F_P(t)$ for the families of associahedra and permutohedra
\beq{asper}
  \begin{aligned}
\mathcal A(t,x):=\sum_{n=1}^\infty F_{\mathfrak a_n}(t)x^{n+1}=\mathfrak a_1 x^2+ (\mathfrak a_2+ 2\mathfrak a_1 t) x^3+(\mathfrak a_3+ 5 \mathfrak a_2 t+ 5 \mathfrak a_1 t^2) x^4 +\dots, \\
 \mathcal P(t,x):=\sum_{n=1}^\infty F_{\pi_n}(t)\frac{x^{n}}{n!}= \pi_1 x +(\pi_2 +2 \pi_1 t)\frac{x^{2}}{2!}+(\pi_3+6\pi_2 t + 6\pi_1 t^2)\frac{x^{3}}{3!}+\dots.
  \end{aligned}
\eeq

 Here both $\mathfrak a_1$ and $\pi_1$ are points and thus are units in the polytopal ring, but for our purposes we will keep track of them considering them as independent variables. In particular, we define the corresponding {\it weighted motivic interiors}  as
\beq{motivic_weight}
  \begin{aligned}
\mathfrak a_n^\circ:=F_{\mathfrak a_n}(-\mathfrak a_1)=\sum_{f\in \mathcal F(\mathfrak a_n)}  (-\mathfrak a_1)^{\text{codim} f} f, \\
\pi_n^\circ:=F_{\pi_n}(-\pi_1)=\sum_{f\in \mathcal F(\pi_n)}  (-\pi_1)^{\text{codim} f} f.
  \end{aligned}
  \eeq
For example, for $n=3$ we have 
$
\mathfrak a_3^\circ=\mathfrak a_3- 5 \mathfrak a_2 \mathfrak a_1+ 5 \mathfrak a_1 ^3,
\,\, \pi_3^\circ=\pi_3-6\pi_2 \pi_1 + 6\pi_1^3.
$
The corresponding motivic interiors (\ref{motivic}) can be found from the weighted motivic interiors by setting $\mathfrak a_1=\pi_1=1.$

Denote by $$u_t(t,x):=\frac{\partial}{\partial t} u(t,x), \,\, u_x(t,x):=\frac{\partial}{\partial x} u(t,x)$$ the corresponding partial derivatives of the function $u(t,x).$

\begin{Theorem} [\cite{B-2008, BP}]
The generating functions (\ref{asper}) satisfy the partial differential equations in $\mathfrak P[t][[x]]$
\beq{as}
\mathfrak a_1 \mathcal A_t=\mathcal A \mathcal A_x, \quad \mathcal A(0,x)=A(x):=\sum_{n=1}^\infty \mathfrak a_{n} x^{n+1},
\eeq
\beq{per}
\pi_1 \mathcal P_t= \mathcal P^2, \quad \mathcal P(0,x)=P(x):=\sum_{n=1}^\infty \pi_n \frac{x^{n}}{n!}.
\eeq
\end{Theorem}
Indeed, the boundary formula (\ref{boundary}) in this particular case gives
\beq{asper1}
d \mathfrak a_n=\sum_{i=1}^{n-1} (i+1) \mathfrak a_i  \times \mathfrak a_{n-i}, \quad d \pi_n=\sum_{i=1}^{n-1}  {n \choose i}\pi_i \times \pi_{n-i},
\eeq
which together with Lemma leads to the differential equations (\ref{as}), (\ref{per}). 


The partial differential equation
\beq{Hopf}
u_t=u u_x
\eeq
is called {\it Hopf equation} (or {\it inviscid Burgers' equation}) and used as the simplest model to describe the breaking of waves phenomenon.

It is well-known that its solution $u(t,x)$ with initial value $u(0,x)=f(x)$  can be given implicitly by the formula
\beq{sol}
u=f(x+ut),
\eeq
which can be easily checked directly. 
Applying this now for the equation (\ref{as}), we have the identity
\beq{solas}
\mathcal A(\mathfrak a_1 t,x)=A(x+\mathcal A(\mathfrak a_1 t,x) t).
\eeq
Substituting here $t=-1$ we see that the generating function
\beq{asso}
A^\circ(x):=\mathcal A(-\mathfrak a_1,x)=\sum_{n=1}^\infty \mathfrak a^\circ_nx^{n+1}
\eeq
satisfies the relation
\beq{solas}
A^\circ(x)=A(x-A^\circ(x)).
\eeq
Consider the series $\xi=x-A^\circ(x)$, then $$x=\xi+A^\circ(x)=\xi+A(\xi).$$ 

Thus we have proved the following theorem, which is an interpretation of the result of Loday \cite{Loday} (see also \cite{AA-2023}).

\begin{Theorem} The generating power series of associahedra and their  weighted motivic interiors
\beq{associ}
\alpha(x)=x+\sum_{n=1}^\infty \mathfrak a_{n}x^{n+1}, \quad \beta(x)=x-\sum_{n=1}^\infty \mathfrak a^\circ_n x^{n+1} 
\eeq
are compositionally inverse to each other in the polytopal ring $\mathfrak P[[x]]$:
$$
\alpha(\beta(x))\equiv x, \quad \beta(\alpha(x)) \equiv x.
$$
\end{Theorem}

Note that since by Theorem 2.1 and Proposition 2.2 the associahedra are algebraically independent in $\mathfrak P$, this implies the following {\it universal inversion formulae}:
the coefficients of the formal series $\alpha(x)=x+\sum_{n\geq 1}a_n x^{n+1}$ and its compositional inverse $\beta(x)=x+\sum_{n\geq 1}b_n x^{n+1}$ are related by the following formulae, 
expressing each other as certain polynomials with integer coefficients:
$$
b_1=-a_1, \quad b_2=-a_2+2a_1^2, \quad b_3=-a_3+5a_1a_2-5a_1^3,$$
$$
b_4=-a_4+6a_1a_3+3a_2^2-21a_1^2a_2+14a_1^4,
$$
in agreement with the formulas for the weighted motivic interiors of the associahedra
$$
\mathfrak a^\circ_1=\mathfrak a_1, \mathfrak a^\circ_2=\mathfrak a_2-2\mathfrak a_1^2, \quad \mathfrak a^\circ_3=\mathfrak a_3-5\mathfrak a_1\mathfrak a_2+5\mathfrak a_1^3,$$
$$
\mathfrak a^\circ_4=\mathfrak a_4-6\mathfrak a_1\mathfrak a_3-3\mathfrak a_2^2+21\mathfrak a_1^2\mathfrak a_2-14\mathfrak a_1^4.
$$

Note that the faces of $\mathfrak a_n$ of codimension $d$ are in one-to-one correspondence with
dissections of a based $(n+3)$-gon by $d$ non-intersecting diagonals (see \cite{GKZ, DR}). Alternatively, the faces of $\mathfrak a_n$ can be labelled by non-isomorphic planar rooted trees with $n+1$ leaves \cite{Devadoss}, which can be used to rewrite the inversion formulas in these terms (see \cite{Stanley}).

The corresponding polynomials can also be expressed in terms of the {\it partial Bell polynomials} $B_{n,k}$ (see e.g. \cite{Comtet}, Section 3.8)
or, more explicitly, as
\beq{invb}
b_n=\frac{1}{(n+1)!} \sum_{k_1+2k_2+\dots=n, \, k_i\geq 0} (-1)^k\frac{(n+k)!} {k_1!k_2!\dots}a_1^{k_1}a_2^{k_2}\dots,
\eeq
where $k=k_1+k_2+\dots$ (see e.g. \cite{Gessel}).

%
%

Similarly, the corresponding permutohedral differential equation (\ref{per})
 has the explicit solution
\beq{solper}
\mathcal P(\pi_1 t,x)=\frac{\pi(x)}{1-t \pi(x)}.
\eeq

Substituting here $t=-1$ we have the relation
$
\mathcal P(-\pi_1,x)=\frac{\pi(x)}{1+\pi(x)}.
$
This implies that
$
1-\mathcal P(-\pi_1,x)=\frac{1}{1+\pi(x)}
$
and the following result, which is an interpretation of the known result from \cite{AA-2023}.

\begin{Theorem}  The exponential generating series of permutohedra and their weighted motivic interiors
\beq{permut_inv}
P(x)=1+\sum_{n=1}^\infty \pi_{n} \frac{x^{n}}{n!}, \quad P^\circ(x)=1-\sum_{n=1}^\infty \pi_{n}^\circ \frac{x^{n}}{n!}
\eeq
is multiplicatively inverse to each other
 in the polytopal ring $\mathfrak P[[x]]$:
$$
P(x) P^\circ(x)\equiv 1.
$$

\end{Theorem}
%
%
 Again since permotohedra are algebraically independent in $\mathfrak P,$ this implies the {\it universal multiplicative inversion formulae} for the formal power series. Namely, the coefficients of the series $p(x)=1+\sum_{n=1}^\infty p_{n} \frac{x^{n}}{n!}$ and its multiplicative inverse $q(x)=1+\sum_{n=1}^\infty q_{n} \frac{x^{n}}{n!}$ are related by the polynomial  formulas with integer coefficients: $q_1=-p_1,\, q_2=-p_2+2p_1^2,$
$$
q_3=-p_3+6p_1p_2-6p_1^3, \,\,\, q_4=-p_4+8p_1p_3+6p_2^2-36p_1^2p_2+24p_1^4,
$$
in agreement with the fact that $\pi_3$ is hexagon and $\pi_4$ has 8 hexagonal and 6 quadrilateral faces, 36 edges and 24 vertices (see Fig. 2). 
This is also in a good agreement with the formulas for the weighted motivic interiors of the permutohedra: $\pi^\circ_1=-p_1,\, \pi^\circ_2=-p_2+2p_1^2,$
\beq{multinv}
 \pi^\circ_3= \pi_3-6 \pi_1 \pi_2+6 \pi_1^3, \,\,\,  \pi^\circ_4= \pi_4-8 \pi_1 \pi_3-6 \pi_2^2+36 \pi_1^2 \pi_2-24 \pi_1^4.
\eeq

\section{Inversion formula and the Deligne-Mumford moduli spaces}

There is an interesting link with the moduli space $\bar M_{n}=\bar M_{0,n}$ of stable genus zero curves with $n$ ordered marked points. Its real version $\bar M_{n+1}(\mathbb R)$ is known to be tessellated by $n!/2$ copies of associahedron $K_n$ (see \cite{Kap,Devadoss}). This explains the result of McMullen, who interpreted the compositional inversion formulas in terms of the corresponding real strata  \cite{McMullen}. 

It is interesting that the geometry of the complex version $\bar M_{n+1}=\bar M_{n+1}(\mathbb C)$ allows also to describe the compositional inversion of the {\it exponential} power series. More precisely, McMullen proved that the compositional inverse of the formal series
$
f(x)=x-\sum_{n\geq 1}a_n\frac{x^{n+1}}{(n+1)!}
$
is given by 
\beq{MMinv}
g(x)=x+\sum_{n\geq 1}b_n\frac{x^{n+1}}{(n+1)!}, \quad b_n=\sum N_{n_1,\dots, n_s}a_{n_1}\dots a_{n_s},
\eeq
where  $N_{n_1,\dots, n_s}$ is the number of strata $S \subset  \bar M_{n}(\mathbb C)$ isomorphic to $$M_{n_1}(\mathbb C)\times \dots \times  M_{n_s}(\mathbb C)$$ (see Theorem 1 in \cite{McMullen}).  The proof is based on the well-known bijection between the strata and the marked stable rooted trees, known also as modular graphs \cite{LZ}.

One can combine this with our result \cite{BV-2020} claiming that the generating functions of the cobordism classes of the complex projective spaces $\mathbb CP^n$ and the theta-divisors $\Theta^n$
$$
\alpha(z)=z+\sum_{n=1}^\infty[\mathbb CP^n]\frac{z^{n+1}}{n+1}, \quad \beta(z)=z+\sum_{n=1}^\infty[\Theta^n]\frac{z^{n+1}}{(n+1)!},
$$
are compositionally inverse to each other, to express the cobordism class $[\Theta^{n-1}]$ in terms of the cobordism classes $[\mathbb CP^k], \, k\leq n-1$ as follows
\beq{thetainv}
[\Theta^{n}]=\sum N_{n_1,\dots, n_s}a_{n_1}\dots a_{n_s}, \quad a_k=-k![\mathbb CP^{k}].
\eeq

One more remarkable fact here is due to Getzler \cite{Getzler}. Recall that $\bar M_{n}$ is a special (Deligne-Mumford) compactification of the moduli space $M_{n}$ of $n$ ordered distinct points on complex projective line considered up to projective equivalence. Using the framework of the operads, Getzler proved that the exponential generating functions of the corresponding (in case of $M_n$, motivic) Poincare polynomials  
$$
F(t,x):=x-\sum_{n\geq 2}P_{M_{n+1}}(t)\frac{x^n}{n!}, \quad G(t,y):=y+\sum_{n\geq 2}P_{\bar M_{n+1}}(t)\frac{y^n}{n!}
$$
are inverse to each other: $F(G(t,y),t)\equiv y,\,\, G(F(t,x),t)\equiv x.$
Since the motivic Poincare polynomial 
\beq{motP}
P_{M_n}(t)=(t^2-2)(t^2-3)\dots (t^2-n+2)
\eeq
 is known explicitly (see e.g. \cite{KLP}), this determines $P_{\bar M_n}(t).$
Substituting here $t=-1$ we have the same relation between exponential generating functions of the Euler characteristics
$$
F(x):=x-\sum_{n\geq 2}\chi(M_{n+1})\frac{x^n}{n!}, \quad G(y):=y+\sum_{n\geq 2}\chi(\bar M_{n+1})\frac{y^n}{n!}.
$$
Since $\chi(M_{n+1})=(-1)^{n}(n-2)!$ we see that $G(y)$ is the compositional inverse of the series
$
F(x)=x-\sum_{n\geq 2}\frac{(-1)^n }{n(n-1)} x^n,
$
 which allows to compute $\chi(\bar M_{n}),$ giving OEIS sequence A074059:
$2,\,7,\,34,\,213,\,1630,\,14747,\,153946,\,1821473...$


It is interesting that for the real version of Deligne-Mumford spaces $\bar M_n(\mathbb R)$ McMullen proved that the same relation $$B_n=\sum N_{n_1,\dots, n_s}A_{n_1}\dots A_{n_s}$$  describes the functional inversion $G(x)=x-\sum_{n\geq 1}B_n x^{n+1}$ of the {\it usual} power series $F(x)=x+\sum_{n\geq 1}A_n x^{n+1}$ (see Corollary 5 in \cite{McMullen}).

Over finite fields one can interpret these results within the approach of Weil and Deligne as follows.

Let $\mathbb F_p$ with $p$ prime be the finite field with $p$ elements and $M_{0,n}(\mathbb F_p), \overline{M_{0,n}(\mathbb F_p)}$ be the corresponding varieties over $\mathbb F_p.$
Let $|M_{0,n}(\mathbb F_p)|$ and $|\overline{M_{0,n}(\mathbb F_p)}|$ be the number of points in these varieties respectively.

\begin{Theorem} The generating power series of the number of points in $M_{0,n}(\mathbb F_p)$ and  $\overline{ M_{0,n}(\mathbb F_p)}$
\beq{M0nFp}
\alpha(x)=x+\sum_{n=2}^\infty |\overline{M_{0,n+1}(\mathbb F_p)}| \frac{x^n}{n!}, \quad \beta(x)=x-\sum_{n=2}^\infty |M_{0,n+1}(\mathbb F_p)| \frac{x^n}{n!} 
\eeq
are compositionally inverse to each other for all primes $p.$
\end{Theorem}

\begin{proof} 
The number of points in $M_{0,n}(\mathbb F_p)$ is easy to compute directly:
$$
|M_{0,n}(\mathbb F_p)|=(p-2)(p-3)\dots (p-n+2), \quad n\geq 4,
$$
with $|M_{0,3}(\mathbb F_p)|=1.$ 

The number of points in $\overline{M_{0,n}(\mathbb F_p)}$ was computed by Amburg, Kreines and Shabat in \cite{AKS}, who proved, in particular, that this number coincides with the value of the Poincare polynomial $P_{\overline{M_{0,n}(\mathbb C)}}(t)$ when $p=t^2$:
$$
|\overline{M_{0,n}(\mathbb F_p)}|=P_{\overline{M_{0,n}(\mathbb C)}}(t), \quad p=t^2.
$$
Since motivic Poincare polynomial of $M_{0,n}(\mathbb C)$ is given by (\ref{motP}), we have the same relation for $M_{0,n}(\mathbb C)$:
$
|M_{0,n}(\mathbb F_p)|=P_{M_{0,n}(\mathbb C)}(t), \quad p=t^2.
$
Now the proof follows from Getzler's result.
\end{proof}

\section{Inversion formula and permutohedral varieties}

The permutohedral variety $X_\Pi^n$ is the toric variety \cite{CLS}, corresponding to the $n$-dimensional permutohedron $\pi_{n+1}$.
In particular, $X_\Pi^1=\mathbb CP^1$, $X_\Pi^2$ is the degree 6 del Pezzo surface.

Every toric variety $X^n$ is a closure of the algebraic torus $\mathbb T_{\mathbb C}^n=(\mathbb C^*)^n$, which can be defined as the {\it motivic interior of $X^n$.}

To justify this consider the {\it Grothendieck ring of complex quasi-projective varieties}  $K_0(\mathcal V_\mathbb C)$ generated by the isomorphism classes of complex quasi-projective
varieties modulo the relations $$[X] = [Y ]+[X \setminus Y ],$$ where $Y$ is a Zariski locally
closed subset of $X$, with multiplication given by the formula $[X]·[Y ] = [X \times Y]$ (see e.g. \cite{GLM-2004}).

Let $X_P$ be the smooth toric variety corresponding to a Delzant polytope $P$ and $Y_f \subset X_P$ be the toric subvariety corresponding to the face $f \in \mathcal F(P).$

The following result provides more justification for our definition of the motivic interior. 
In analogy with (\ref{motivic}) define the {\it motivic interior of $X_P$} as
\beq{motivic2}
[X_P^\circ]:=\sum_{f\in \mathcal F(P)}  (-1)^{\text{codim} f} [Y_f] \in K_0(\mathcal V_\mathcal C).
\eeq

\begin{prop} In the Grothendieck ring $K_0(\mathcal V_\mathbb C)$ we have the relation
\beq{motivicG}
\sum_{f\in \mathcal F(P)}  (-1)^{\text{codim} f} [Y_f] =[\mathbb T_{\mathbb C}^n].
\eeq
In other words, the motivic interior of toric variety $X_P$ is $[X_P^\circ]=[\mathbb T_{\mathbb C}^n].$ 
\end{prop}

The proof follows from the Proposition-definition 6 in \cite{PukhKhov} reformulated in terms of toric geometry \cite{CLS} (see also Theorem 3.2.4 in \cite{Nicaise}).

Note that similar relation holds for all convex polytopes if we extend it to the framework of the toric topology \cite{BP}.

This implies that Theorem 3.5 can be reformulated as follows.

\begin{Theorem} The exponential generating series of permutohedral varieties and algebraic tori
\beq{permut_motiv}
\pi(x):=1+\sum_{n=1}^\infty [X_\Pi^n] \frac{x^{n}}{n!}, \quad \tau(x):=1-\sum_{n=1}^\infty [\mathbb T_{\mathbb C}^{n-1}] \frac{x^{n}}{n!}
\eeq
are multiplicatively inverse to each other in the ring $K_0(\mathcal V_\mathbb C)[[x]].$
\end{Theorem}

{\it Motivic Poincare polynomial} $P_Z(t), \, Z \in K_0(\mathcal V_\mathbb C)$ defines the ring homomorphism $K_0(\mathcal V_\mathbb C) \to \mathbb Z[t],$
where for the projective varieties $X,Y\subseteq X$ by definition $$P_{[X\setminus Y]}(t)=P_X(t)-P_Y(t),$$ where $P_X(t)$ is the usual Poincare polynomial.
In particular, since $\mathbb C^*$ is the complement of two points in $\mathbb CP^1$, the motivic Poincare polynomial
$P_{[\mathbb C^*]}(t)=P_{\mathbb CP^1}(t)-2=(t^2+1)-2=t^2-1,$ and thus
$$
P_{\mathbb T_{\mathbb C}^n}(t)=(t^2-1)^n.
$$

As a  corollary we have the following well-known fact (see e.g. \cite{PRW}) relating permutohedral varieties with the classical {\it Eulerian polynomials} \cite{Carlitz}.
These polynomials were introduced by Euler in 1755 by the relation
$$
\sum_{k=1}^\infty k^nt^n=\frac{tA_n(t)}{(1-t)^{n+1}}.
$$
They have the generating function
\beq{genep}
\sum_{n\geq 0}A_n(s)\frac{x^n}{n!}=\frac{s-1}{s-e^{(s-1)x}}
\eeq
and can be computed recursively by
$$
A_{n+1}(t)=[t(1-t)\frac{d}{dt} +nt+1]A_n(t), \quad A_0=A_1=1:
$$
$$A_1=1, \,\, A_2=s+1,\,\, A_3=s^2+4s+1,\,\,
A_4=s^3+11s^2+11s+1,$$
$$A_5=s^4+26s^3+66s^2+26s+1, \,\, A_6=s^5+57s^4+302s^3+302s^2+57s+1.
$$
Their coefficients are known as Eulerian numbers and have natural combinatorial interpretations (see \cite{Stanley}).

\begin{Corollary} The Poincare polynomial of permutohedral varieties is
\beq{poinper}
P_{X_\Pi^n}(t)=A_n(t^2),
\eeq
where $A_n(s)$ are the Eulerian polynomials.
\end{Corollary}

\begin{proof}
The generating function of the motivic Poincare polynomials of the algebraic tori can be computed explicitly as
$$
\sum_{n=0}^\infty P_{\mathbb T_{\mathbb C}^n}(t)\frac{x^{n+1}}{(n+1)!}=\sum_{n=0}^\infty (t^2-1)^n\frac{x^{n+1}}{(n+1)!}=\frac{e^{(t^2-1)x}-1}{t^2-1}.
$$
This means that
$$1-\sum_{n=0}^\infty P_{\mathbb T_{\mathbb C}^n}(t)\frac{x^{n+1}}{(n+1)!}==\frac{t^2-e^{(t^2-1)x}}{t^2-1},$$
so from Theorem 5.1 we have
$$
1+\sum_{n=0}^\infty P_{X_\Pi^n}(t) \frac{x^{n+1}}{(n+1)!}=\frac{t^2-1}{t^2-e^{(t^2-1)x}}.
$$
Comparing this with (\ref{genep}), we have the claim.
\end{proof}

{\bf Remark.} It is interesting to mention that the permotuhedral variety $X_\Pi^n$ is isomorphic to the {\it Losev-Manin} \cite{LM} compactification $\bar L_{0,n+3,2}$ of the moduli space $M_{0,n+3}$ (see more on this in the recent paper \cite{BT-2024}).

\section{Cyclohedra and Fa\`a di Bruno's formula}

There is another remarkable family of the graph-polyhedra: {\it cyclohedra} (or {\it Bott-Taubes polytopes}) $W_n$, corresponding to the $n$-cycle graph. 
In particular, $W_3$ is hexagon and $W_4$ is $3D$ polyhedron shown on Fig. 3.

These polytopes first appeared in Bott and Taubes \cite{BT} in connection with the link invariants, although implicitly they were already in the earlier work by Kontsevich \cite{Kont} (see more detail in \cite{Dev2}). 

\begin{figure}[h]
\includegraphics[width=35mm]{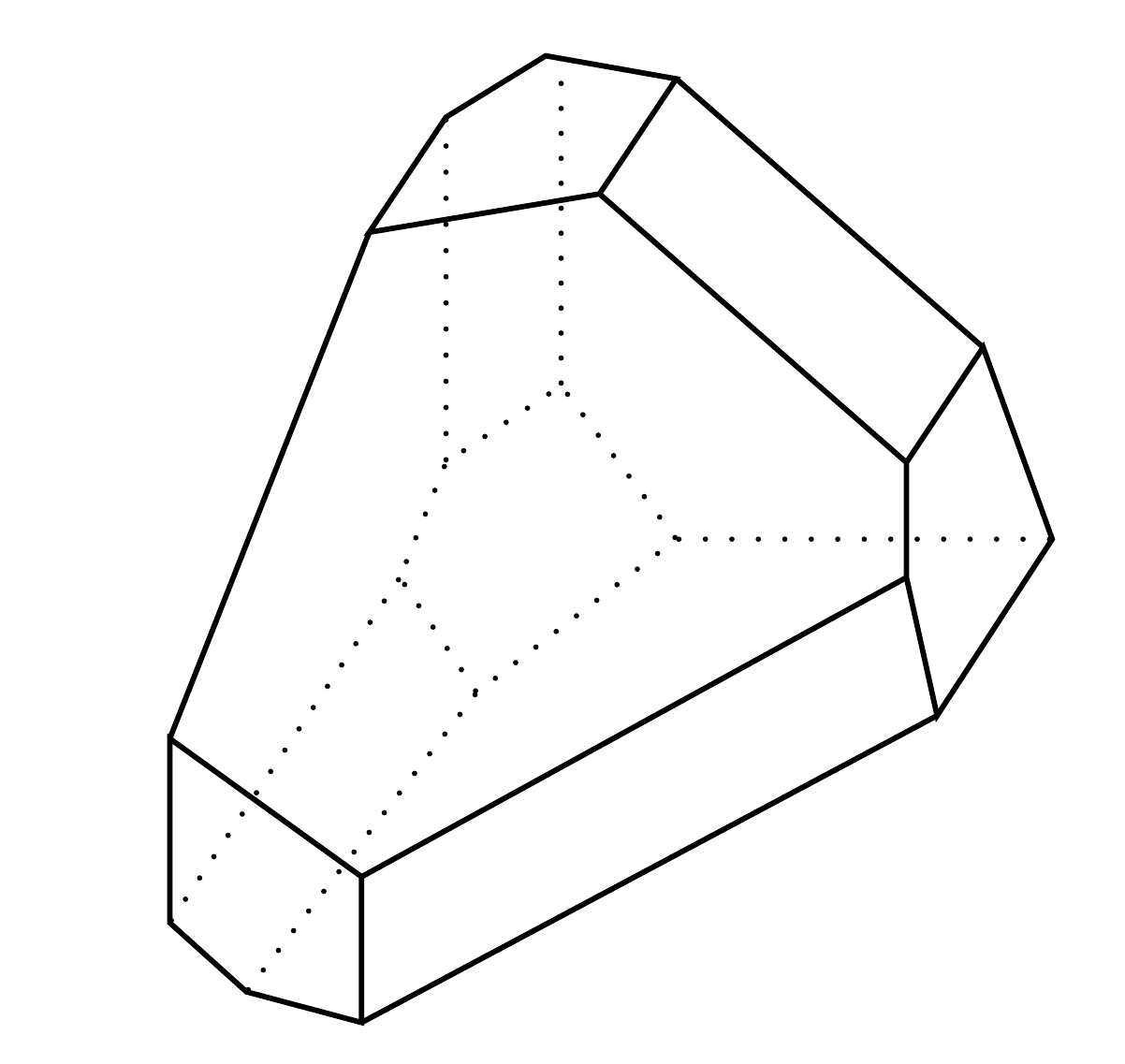} 
\caption{3D cyclohedron $W_4$.}
\end{figure} 

Its combinatorics was studied by Simion \cite{Simion}, who had shown, in particular, that the number $f_k(W_{n+1})$ of faces of dimension $k$ of $W_{n+1}$ equals
$$
f_k(W_{n+1})={2n-k \choose n-k}{n \choose k}.
$$
In particular, $W_{n+1}$ has ${2n \choose n}$ vertices and $n(n+1)$ facets.
The corresponding boundary formula
$
d W_n=n\sum_{i=1}^{n-1} \mathfrak a_{i}\times W_{n-i}
$
implies that the generating function
$
\mathcal C(t,x):=\sum_{n=1}^\infty F_{W_n}(t)\frac{x^{n}}{n},
$
satisfies the following non-autonomous linear differential equation
$$
 \mathcal C_t=\mathcal A(t,x)\mathcal C_x, \quad  \mathcal C(0,x)=W(x):=\sum_{n=1}^\infty W_n\frac{x^{n}}{n},
$$
where $\mathcal A(t,x)$ are the corresponding function for the associahedra (\ref{asper}) and we set for simplicity $\mathfrak a_1=W_1=1$. 
 It is best to combine these two functions together as the solution of the following system of PDEs
\beq{system}
 \mathcal A_t=\mathcal A \mathcal A_x, \quad
 \mathcal C_t=\mathcal A\mathcal C_x.
\eeq
The solution of this system with the prescribed initial values is given by (\ref{solas}) and
$$
\mathcal C(t,x)=W(x+t\mathcal A(t,x)).
$$
Substituting here $t=-1$ we have the following result. Let $W^\circ_n$ be the motivic interior (\ref{motivic}) of the cyclohedron $W_n.$

\begin{Theorem} The generating function of the  motivic interiors of the cyclohedra
\beq{cycl}
W^\circ(x):=\sum_{n=1}^\infty \frac{W^\circ_n}{n} x^n
=W(\beta(x))
\eeq
is the result of the substitution into $W(x)$ the series $\beta(x)$, which is the compositional inverse of the associahedral series $\alpha(x)$ given by (\ref{associ}).
\end{Theorem}

The result of the substitution of formal series can be given by the following combinatorial formula attributed to Fa\`a di Bruno (see the history in \cite{Johnson}).
It has several important interpretations within the theory of Lie and Hopf algebras and operads \cite{FM}, as well as relations with integrable systems \cite{ShabatEfendiev}.

For the formal power series {\it Fa\`a di Bruno formula} has the following form.
Let $f(x)=\sum_{n\geq 1}a_nx^n, \,\, g(x)=\sum_{n\geq 1}b_nx^n,$ then the coefficients of the composition
$f(g(x))=\sum_{n\geq 1}c_nx^n$
can be written as
\beq{FdB}
c_n=\sum {\frac {k!}{k_{1}!\,k_{2}!\,\cdots \,k_{n}!}}a_k\prod_{j=1}^nb_j^{k_j},
\eeq
where $k=k_1+\dots+k_n$ and the summation is taken over all integer $(k_1,\dots, k_n), k_i\geq 0$ such that $k_1+2k_2+\dots+nk_n=n.$  
In particular, we have
$$
c_1=a_1b_1, \,\, c_2=a_1b_2+a_2b_1,\,\, c_3=a_1b_3+2a_2b_1b_3+a_3b_1^3,
$$
$$
c_4=a_1b_4+a_2(b_2^2+2b_1b_3)+3a_3b_1^2b_2+a_4b_1^4.
$$
Combining this with the inversion formula (\ref{invb}), one can compute the coefficients $d_n$ of the function
$$h(x)=f(g^{(-1)}(x))=\sum_{n\geq 1}d_nx^n,$$ assuming for convenience that $a_1=b_1=1$:
$$
d_1=1, \,\, d_2=a_2-b_2,\,\, d_3=a_3-2a_2b_2+2b_2^2-b_3,
$$
$$
d_4=a_4-3a_3b_3+a_2(5b_2^2-2b_3)+-5b_2^3+5b_2b_3-b_4.
$$

The combinatorial formulas for the corresponding modification of the Bell polynomials $d_n=D_n(a,b)$ in terms of the so-called pointed non-crossing partitions can be found in \cite{Bast} (see Theorem 5.8). 

As a corollary we have the following polytopal interpretation of the corresponding polynomials, which was first discovered by Aguiar and Bastidas in relation with the Fa\`a di Bruno Hopf monoid \cite{Bast} (see Theorem 5.6). Note that both composition and inversion of formal series are naturally embedded into the so-called Fa\`a di Bruno Hopf algebra, playing important role in various problems of mathematics and physics, see \cite{Mvondo} and references therein. 

\begin{Corollary}
The motivic interior of the cyclohedron $W_n$ can be written in the algebra of polytopes as
\beq{cyc}
W^\circ_n =n D_n(a,b), \quad a_j=\frac{W_j}{j}, \, b_j=\mathfrak a_{j-1}.
\eeq
\end{Corollary}

In particular, we have the formulas
$$
W^\circ_1=W_1, \,\, W^\circ_2=W_2-2\mathfrak a_1,\,\, W^\circ_3=W_3-3W_2\mathfrak a_1-3\mathfrak a_2+6\mathfrak a_1^2,
$$
$$
W^\circ_4=W_4-4W_3\mathfrak a_1+10 W_2\mathfrak a_1^2 -4W_2\mathfrak a_2-4W_1\mathfrak a_3+20W_1 \mathfrak a_1 \mathfrak a_2-20W_1\mathfrak a_1^3.
$$
This is in agreement with the fact that $W_3$ is a hexagon and $W_4$ has 4 hexagonal, 4 pentagonal and 4 quadrilateral faces, 30 edges and 20 vertices (see Fig. 3).
Again since cyclohedra are algebraically independent, these formulae are universal.

\section{Equation and formulas for the stellohedra}

There is another interesting graph-associahedron called {\it stellohedron} $S_n$, corresponding to the star-graph with $n$ vertices (see Fig.4). In particular, $S_3=\mathfrak a_3$ is a pentagon and $S_4$ is the polyhedron shown on Fig.4. For more details about combinatorics of stellohedra we refer to \cite{B-2008,PRW}.

\begin{figure}[h]
\includegraphics[width=35mm]{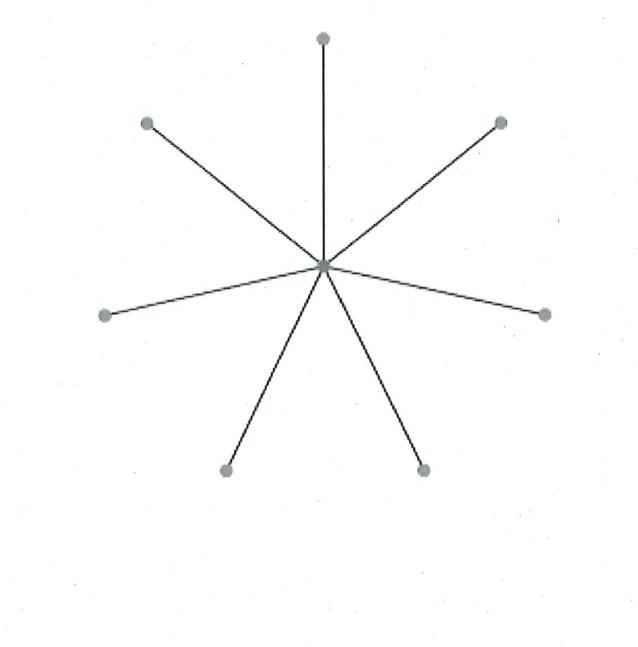} \quad \quad \includegraphics[width=35mm]{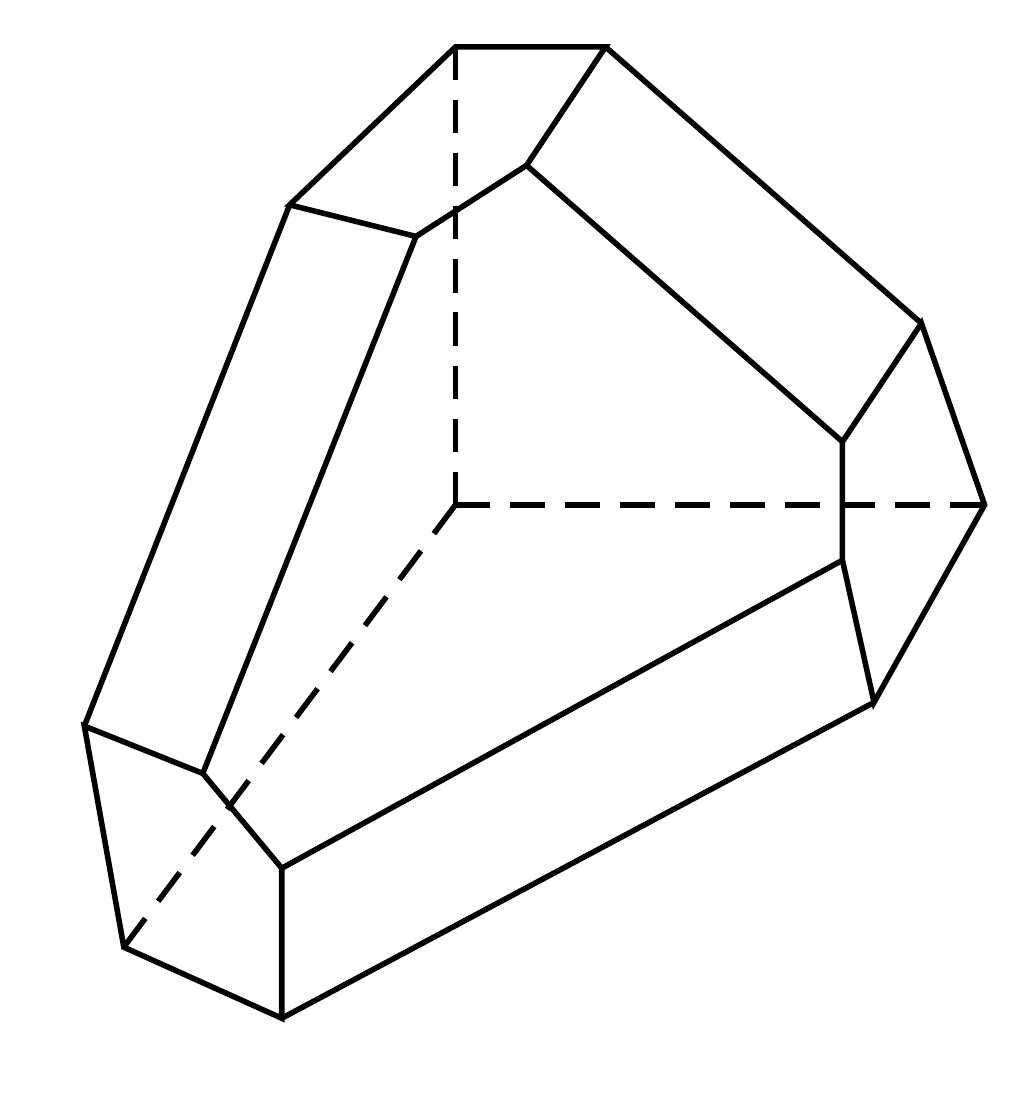} 
\caption{Star-graph and 3D stellohedron $S_4$.}
\end{figure} 

The corresponding boundary formula
$$
 d S_n=(n-1) S_{n-1}+\sum_{i=1}^{n-1}  {n-1 \choose i}S_i \times \pi_{n-i}
$$
implies that the corresponding exponential generating functions of stellohedra and permutohedra
$$ \mathcal{S} (t,x):=\sum_{n=0}^\infty F_{S_{n+1}}(t)\frac{x^{n}}{n!}, \quad  \mathcal P(t,x):=\sum_{n=1}^\infty F_{\pi_n}(t)\frac{x^{n}}{n!}
$$
satisfy the following system of equations
\beq{system}
\mathcal P_t=\mathcal P^2, \quad
\mathcal S_t= (x+ \mathcal P) \mathcal S.
\eeq
For given initial data $\mathcal P(0,x)=\Pi(x), \, \mathcal S(0,x)=S(x):=\sum_{n=0}^\infty S_{n+1}\frac{x^{n}}{n!}$ this system can be easily solved explicitly as
$$
\mathcal P(t,x)=\frac{\Pi(x)}{1-t \Pi(x)}, \quad \mathcal S(t,x)=\frac{e^{tx} S(x)}{1-t \Pi(x)}
$$
(see \cite{B-2008,BP}).
Substituting here $t=-1$ and using theorem 3.5 we have  

\begin{Theorem} The exponential generating function of the  motivic interiors of the stellohedra can be expressed via permutohedral generating functions (\ref{permut_inv}) by the formula
\beq{stelloh}
S^\circ(x):=\sum_{n=0}^\infty S^\circ_{n+1}\frac{x^n}{n!}=e^{-x}\Pi^\circ(x) S(x)=e^{-x}\Pi^{-1}(x)S(x).
\eeq
\end{Theorem}
In particular, using formulas (\ref{multinv}) we have
$$
S^\circ_4=S_4-3S_3(1+\pi_1)+3S_2(1+2\pi_1+2\pi_1^2-\pi_2)
$$
$$-S_1(1+3\pi_1+6\pi_1^2+6\pi_1^3-3\pi_2-6\pi_1\pi_2+\pi_3)
$$
in agreement with the fact that 3D stellohedron has 1 hexagonal, 6 pentagonal and 3 quadrilateral faces, 24 edges and 16 vertices (see Fig. 4).

\begin{Corollary}
The exponential generating function of the numbers of vertices $V_n$ of the stellohedra $S_n$ has the form
 \beq{cycver}
\sum_{n=0}^\infty V_{n+1}\frac{x^{n}}{n!} =\frac{e^x}{1-x}.
\eeq
\end{Corollary}

In particular, this implies that the number of vertices is
$$
V_{n}=\sum_{k=0}^{n-1}\frac{(n-1)!}{k!}.
$$
More general formulas, expressing the number of faces of stellohedra of any dimension, can be found in \cite{B-2008} (Section 9) and \cite{PRW} (Section 10.4).

For completeness consider also the families of simplices with $n$ vertices $\Delta_n$ and $n$-dimensional cubes $\Box_n$
and define
$$
\Delta(x)=\sum_{n\geq 1} \Delta_n \frac{x^n}{n!}, \quad \Box(x)=\sum_{n\geq 0} \Box_n \frac{x^n}{n!}.
$$
The corresponding exponential generating functions 
$$ \Delta (t,x):=\sum_{n\geq 1} F_{\Delta_n}(t)\frac{x^{n}}{n!}, \quad  \Box(t,x):=\sum_{n\geq 0} F_{\Box_n}(t)\frac{x^{n}}{n!}
$$
satisfy the differential equations \cite{BP}
$$
 \Delta_t(t,x)= x \Delta(t,x), \quad
\Box_t(t,x)=2x \Box(t,x).
$$
 Solving them explicitly with initial data $\Delta(0,x)=\Delta(x), \, \Box(0,x)=\Box(x)$:
$$
\Delta(t,x)=\Delta(x)e^{tx}, \,\, \Box(t,x)=\Box(x)e^{2tx},
$$
we come to the binomial formulas for the interiors
$$
\Delta^\circ_n=\sum_{k=0}^{n-1}(-1)^k{n \choose k}\Delta_{n-k}, \quad \Box^\circ_n=\sum_{k=0}^{n}(-2)^k{n \choose k}\Box_{n-k}.
$$
It is interesting to note that the polynomials $p_n(t)=F_{\Delta_n}(t)$ form the universal {\it Appell sequence} \cite{Appell}, satisfying the characteristic relation
$
p_n'(t)=np_{n-1}(t)
$
(see \cite{Bateman}, Ch. 19.3).

\section{Concluding remarks}

We have seen that associahedra and permutohedra play a special role in the differential algebra of graph-associahedra, generating subalgebras of the polytopal algebra $\mathfrak P$, which are invariant under the differentiation. This explains their relation with the differential equations and ultimately with the inversion formulas. 

The families of cyclohedra and stellohedra generate the subalgebras of $\mathfrak P$, which can be viewed as the modules over the differential subalgebras generated by associahedra and permutohedra respectively. In the language of operads this was pointed out in \cite{Markl} (see also \cite{Dev2, Stas2}). The description of all such subalgebras and modules over them within the differential algebra of polytopes $\mathfrak P$ seems to be a very interesting open problem. 

The equations which appeared here are easily solvable, which hints the link with the theory of integrable systems. At the combinatorial level this link was already discussed in the literature, see e.g. \cite{Falqui, Lambert,Mvondo,ShabatEfendiev} and references therein. However, a general picture of the relations between combinatorics and integrable systems seems to be quite rich and needs better understanding (see \cite{Adler_1,Adler} for some very interesting thoughts in this direction).

In this relation it is worthy to mention that the Hopf equation $u_t=u u_x$, which appeared in our paper in relation with associahedra, is the dispersionless limit $\varepsilon \to 0$ of the celebrated KdV equation
$$
u_t=u u_x+\varepsilon^2 u_{xxx},
$$
well-known in soliton theory and enumerative algebraic geometry. In particular, according to Witten \cite{Witten} the KdV hierarchy determines a certain generating function of the characteristic numbers of $\bar M_{g,n}$ (see \cite{LZ} for the detail),
which makes the appearance of the inversion formula in the theory of the moduli spaces $\bar M_{0,n}$ (discussed in Section 4) a bit less mysterious.

We would like to comment on the choice of the coefficients $\lambda_n$ in the generating functions $F_P(x)=\sum \lambda_n P_n x^n.$ Most common cases are $\lambda_n=1$ and $\lambda_n=1/n!$, corresponding to the usual and exponential generating functions. However, we have seen that in the case of cyclohedra the most appropriate choice is $\lambda_n=1/n$, related to the usual choice by integration. The role of different choices of $\lambda_n$ in Boas-Buck's approach \cite{BoasBuck} to the generating functions  was emphasized in \cite{BKh}. Note that both composition and multiplication of power series play the central role in this approach, going back to Appell \cite{Appell} and providing umbrella for many classical polynomials. 

Since the graph-associahedra are indecomposable, it is natural to ask if the corresponding $h$-polynomials are irreducible over $\mathbb Q$. 
For example, for the permutohedra the corresponding $h$-polynomials are the Eulerian polynomials $A_n(t)$. It is known that $A_{2k}(t)$ is divisible by $t+1$ (which is $h$-polynomial of the segment), but there is a conjecture that $A_{2k+1}(t)$ and $A_{2k}(t)/(t+1)$ are irreducible over $\mathbb Q$ (see \cite{Heidrich}).
It would be interesting to study similar question for other graph-assiciahedra from our paper.

\section{Acknowledgements.}
We are very grateful to Vsevolod Adler, Alexander Braverman, Nikolai Erokhovets, Alexander Gaifullin and Georgy Shabat for the useful discussions, to Jose Bastidas, who attracted our attention to the preprint \cite{Bast} and to Jim Stasheff for the encouraging comments.

\end{document}